\documentclass[12pt]{article}

\usepackage[paperwidth=8.5in, paperheight=11in, top=1in, bottom=1in, left=.5in, right=.5in]{geometry}

\usepackage{amsmath,amsfonts,amssymb,epsfig,euscript,mathrsfs,graphicx}
\usepackage{xspace,multicol,enumitem}
\usepackage{theorem}
\usepackage{color}
\usepackage{caption, float}
\usepackage[dvipsnames]{xcolor}

\newtheorem{theorem}{Theorem}[section]

\newtheorem{lemma}[theorem]{Lemma}
\newtheorem{question}{Question}

\newcommand{\ds}{\displaystyle}

\newenvironment{proof}[1][Proof]{\begin{trivlist}
\item[\hskip \labelsep {\bfseries #1}]}{\end{trivlist}}

\newcommand{\qed}{\hfill \ensuremath{\Box}}

\begin{document}

\thispagestyle{empty}
%\begin{titlepage}
\title{\textbf{Problems on the Triangular Lattice}}

\author{\textbf{Gaston A. Brouwer\footnote{gaston.brouwer@mga.edu} \hspace{20pt}Jonathan Joe\footnote{jonathan.joe@mga.edu}}\\[10pt]
{\textbf{Abby A. Noble\footnote{abby.noble@mga.edu} \hspace{20pt}Matt Noble\footnote{matthew.noble@mga.edu}}}\\[10pt]
Department of Mathematics and Statistics\\
Middle Georgia State University\\
Macon, GA 31206}

\date{}
\maketitle

\begin{abstract}
In this work, we consider a number of problems defined on the triangular lattice with $n$ rows, which we will denote as $T_n$.  Define a \textit{proper coloring} to be an assignment of colors to the points of $T_n$ such that no three points constituting the vertices of an equilateral triangle all receive the same color, and denote by $f(n)$ the smallest possible number of colors that can be used in a proper coloring of $T_n$.  We either determine exactly or give upper bounds for $f(n)$ for many small values of $n$, and it is shown that $\lim_{n\to\infty} \frac{f(n)}{n} \leq \frac13$.  We also give formulas counting the number of pairs of points in $T_n$ for which there are, respectively, 0, 1, or 2 choices of points in $T_n$ which extend those two into the vertices of an equilateral triangle.  Along the way, we pose a number of related questions.\\

\end{abstract}
%\end{titlepage}

\section{Introduction}

In March of 2019, the fourth author attended the SEICGTCCC in Boca Raton, Florida.  While there, he attended a presentation given by Braxton Carrigan of Southern Connecticut State University.  The talk was ostensibly about a scheduling problem and its graph theoretic solution, but during his discussion, the speaker offhandedly posed the following problem, which we will designate as Question \ref{mainproblem}.

\begin{question} \label{mainproblem} What is the minimum number of colors needed to color the $n$-row triangular lattice such that no three points constituting the vertices of an equilateral triangle each receive the same color? 
\end{question}

Throughout this paper, we will denote by $T_n$ the triangular lattice with $n$ rows, and we will let $f(n)$ be the answer to Question \ref{mainproblem}.  In an effort to be abundantly clear, $T_4$ is drawn in Figure \ref{t4}.

\begin{figure}[h]%
\begin{center}
\includegraphics[scale=.8]{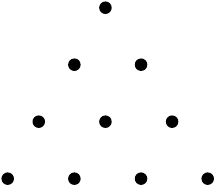}%
\caption{}
\label{t4}%
\end{center}
\end{figure}

Define a \textit{proper coloring} as an assignment of colors to the points of $T_n$ so that no three points that are the vertices of an equilateral triangle receive the same color.  In the parlance associated with this and related subject matter, such a coloring is often said to \textit{forbid} monochromatic equilateral triangles in $T_n$.  Note also that Question \ref{mainproblem} is asking for a coloring that forbids all equilateral triangles and not just those that have one edge being horizontal.

In later conversation, Braxton indicated to us that he had conceived of Question \ref{mainproblem} as a potential student research project.  We will admit to being surprised that it has not received a formal treatment in the literature, as the question feels quite natural.  There have been, however, a few informal mentions along with attacks on similar problems.  In an REU \cite{lyall} conducted by Neil Lyall and Akos Magyar in 2005, Question \ref{mainproblem} is also proposed as a potential topic for student research.  As well, Lyall and Magyar offer a variant of Question \ref{mainproblem} where one is asked to find the minimum number of colors needed in a coloring of $T_n$ that specifically forbids monochromatic equilateral triangles whose bottom side is horizontal.  Unfortunately, it seems that no results were obtained, or at least, none found their way to the literature.  When $n = 4$,  Question \ref{mainproblem} has popped up in several online problem-solving forums, and from the discussion there, it appears as if this specific case has been posed in undergraduate problem-solving competitions.  Indeed, $n=4$ is the first non-trivial case, and we invite readers to take a moment to verify for themselves that $f(4) = 3$.

Forbidden substructure questions are of course ubiquitous in both geometry and combinatorics, and Ramsey-type problems on the integer lattice are well-studied.  The most fundamental result is the theorem independently proven by Gallai and Witt \cite{witt} (the development of which is detailed in \cite{ramsey}) showing that for any finite $S \subset \mathbb{Z}^2$ and any coloring of $\mathbb{Z}^2$ using finitely many colors, there exists a monochromatic subset $S'$ of $\mathbb{Z}^2$ which is homothetic to $S$. Indeed, typical problems in this subject matter begin with a suitable $S$, and then ask for the minimum $n$ such that a coloring of an $n$ by $n$ sub-grid of $\mathbb{Z}^2$ using $k$ colors is guaranteed to have a monochromatic subset homothetic to $S$.  As a sample of notable results, consider results by Graham and Solymosi \cite{solymosi}, by Fenner, et al. \cite{fenner}, and by Axenovich and Manske \cite{axenovich}.

Our study of Question \ref{mainproblem} will be a slight departure from the way these questions are typically analyzed in several ways.  As mentioned above, in $\mathbb{Z}^2$, for a given $S$, colorings are usually sought that forbid monochromatic sets $S'$ where $S'$ is a translation and/or dilation of $S$.  Here, in our coloring of the $n$-row triangular lattice, we are also looking to forbid all rotations of an equilateral triangle.  This difference greatly increases the number of equilateral triangles that can be found in $T_n$, asymptotically by a factor of $\frac{n}{4}$.  Letting $\alpha(n)$ and $\beta(n)$ be the number of $3$-element subsets of $T_n$ that are, respectively, the vertices of an equilateral triangle, or the vertices of an equilateral triangle having its bottom side horizontal, we have $\alpha(n) = \frac{n^4 + 2n^3 - n^2 - 2n}{24}$ (OEIS sequence A000332) and $\beta(n) = \frac{n^3 - n}{6}$ (OEIS sequence A000292).  Also typical with this subject matter is that results are usually sought concerning lower bounds for functions like our $f(n)$.  In Section 2, we will instead give upper bounds for $f(n)$ for small values of $n$, along with exactly determining $f(n)$ for $n \leq 8$.  We also show that $\lim_{n\to\infty} \frac{f(n)}{n} \leq \frac13$.

In Sections 3 and 4, we will consider two additional problems that organically surfaced while we were investigating Question \ref{mainproblem}.  An immediate observation is that for a given pair of points $p_1$ and $p_2$ of $T_n$, there are either zero, one, or two possible equilateral triangles which have $p_1$ and $p_2$ as two of their vertices and their third vertex also being a point of $T_n$.  With this in mind, we define $a_0(n)$, $a_1(n)$, and $a_2(n)$ to be the number of pairs of points in $T_n$ that are vertices of zero, one, or two equilateral triangles in $T_n$.  In Section 3, we derive explicit formulas for $a_0(n)$, $a_1(n)$, and $a_2(n)$, with the surprising result being that $a_0(n) = a_2(n)$ for all $n$.  In Section 4, we put forth an extension of a Steiner triple system, a common structure in design theory corresponding to a decomposition of the complete graph $K_k$ into copies of $K_3$.  In our modified version, given some $k, r$, we ask for the existence of a multigraph $G$ on $k$ vertices where $r$ pairs of vertices are not adjacent and $r$ pairs of vertices are joined by a double edge, and $G$ can be decomposed into triples.  The triangular lattice $T_n$ is an example of such a structure with $k = \frac{n(n+1)}{2}$ and $r = a_0(n) = a_2(n)$.

\section{A Coloring Problem}

We begin this section by exactly determining $f(n)$ for $n \leq 8$.  It is trivial to see that $f(1) = 1$ and that $f(2) = f(3) = 2$.  As mentioned in the previous section, the determination of $f(4)$ can be found in informal online discussion, and for the sake of completion, we include proof here as well.

\begin{theorem} $f(4) = 3$
\end{theorem}    

\begin{proof} It is obvious that $T_4$ can be properly colored with three colors.  Assume to the contrary that $f(4) < 3$, and that $T_4$ has been properly colored with colors red and blue.  In Figure \ref{4rowsproof}, suppose point $x$ has been colored red, and consider the points labeled $v_0, \ldots, v_5$.\\  

\begin{figure}[h]%
\begin{center}
\includegraphics[scale=.9]{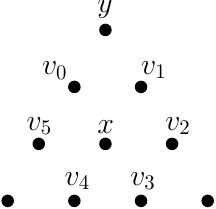}%
\caption{}
\label{4rowsproof}%
\end{center}
\end{figure}

For any $v_i, v_{i+1}$, we cannot have both being red, and as well, we cannot have points $v_0, v_2, v_4$ or $v_1, v_3, v_5$ each being blue.  We conclude that exactly two of $v_0, \ldots, v_5$ are colored red, and without loss of generality, these two red vertices can be taken to be $v_2$ and $v_5$.  Vertex $y$ can then be colored with neither red nor blue.\qed
\end{proof}

In attempting to find proper colorings of $T_n$ for larger $n$, it is helpful to think of $T_n$ as a $3$-uniform hypergraph.  Three vertices induce an edge in the hypergraph if they exactly correspond to the vertices of an equilateral triangle.  Indeed, the second author was able to implement a hypergraph coloring program written in SAGE to find proper $3$-colorings for all $T_n$ with $5 \leq n \leq 8$.  The coloring for $T_8$ is given in the following Figure \ref{8coloring}.\\

\begin{figure}[h]%
\begin{center}
\includegraphics[scale=.8]{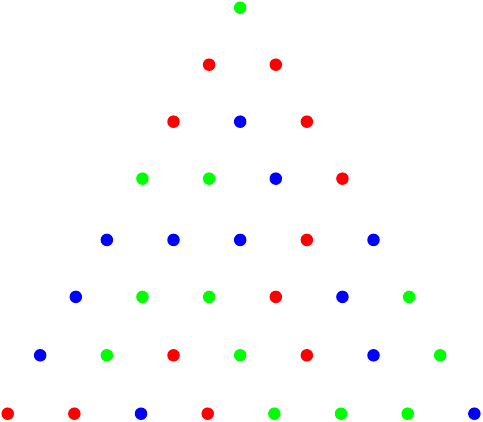}%
\caption{}
\label{8coloring}%
\end{center}
\end{figure}

We strongly suspect that $f(9) = 4$.  Our program did not find a proper $3$-coloring of $T_9$, however we were unable to prove that such a coloring does not exist.  We will make this Question \ref{n=9problem}.

 \begin{question} \label{n=9problem} Is it true that $T_9$ has no proper $3$-coloring? 
\end{question}

Upper bounds were found for other small values of $n$.  They are given below.

\begin{center}

\renewcommand{\arraystretch}{1.5}

%\resizebox{5.3in}{!} {
\begin{tabular}{|p{1.9cm}|p{1.9cm}|p{1.9cm}|p{1.9cm}|}

\hline

$f(9) \leq 4$			& 	$f(12) \leq 5$	&  $f(15) \leq 5$    & $f(18) \leq 7$	 \\ \hline

$f(10) \leq 4$			& 	$f(13) \leq 5$	&  $f(16) \leq 6$    & $f(19) \leq 7$	 \\ \hline

$f(11) \leq 4$			& 	$f(14) \leq 5$	&  $f(17) \leq 6$    & $f(20) \leq 7$	 \\ \hline

\end{tabular}
%}
\end{center}

We will now make some analysis of the asymptotic behavior of $f(n)$, where we begin by noting that $\lim_{n\to\infty} f(n) = \infty$.  Recall van der Waerden's theorem \cite{waerden}, which asserts that for any $r,k$, there exists a corresponding $N$ such that any coloring of the integers $1, \ldots, N$ with $r$ different colors is guaranteed to contain a monochromatic arithmetic progression of length $k$.  The minimum such $N$ is typically denoted as the \textit{van der Waerden number} $W(r,k)$.

\begin{theorem} \label{lowerbound} The limit $\ds \lim_{n\to\infty} f(n)$ diverges to $\infty$.
\end{theorem} 

\begin{proof} Assume to the contrary that for some integer $r$, $f(n) \leq r$ for all $n$.  Furthermore, assume that $r$ is minimum with respect to this property and that $n_0$ is the smallest positive integer such that $f(n_0) = r$.  Consider $T_n$ where $n$ is equal to the van der Waerden number $W(r, n_0 + 1)$, and suppose $T_n$ has been properly colored with $r$ colors.  In the bottom row of this $T_n$, let $P = \{p_1, \ldots, p_{n_0+1}\}$ be the points of a monochromatic arithmetic progression, say, colored using color $c$.  Let $M$ be the set of all points of $T_n \setminus P$ which could be used to form the vertices of an equilateral triangle using two points of $P$.  Note that $M$ is isomorphic to $T_{n_0}$, and thus requires at least $r$ colors in a proper coloring, none of which can be color $c$.  It then follows that $f(n) \geq r + 1$.\qed
\end{proof}

We will specifically consider $\lim_{n\to\infty} \frac{f(n)}{n}$.  First, we note that it is easy to see that $\lim_{n\to\infty} \frac{f(n)}{n} \leq \frac12$, as evidenced by the following coloring.  One color class consists of the points in the middle column of $T_n$.  Other color classes consist of points lying on pairs of lines which make a $60^{\circ}$ angle with the horizontal and which intersect in a point in the middle column.  As an example, such a coloring is done for $T_7$ in Figure \ref{7rowscolored}.  Note that this gives $f(n) \leq \lfloor \frac{n}{2} \rfloor + 1$.\\

\begin{figure}[h]%
\begin{center}
\includegraphics[scale=.9]{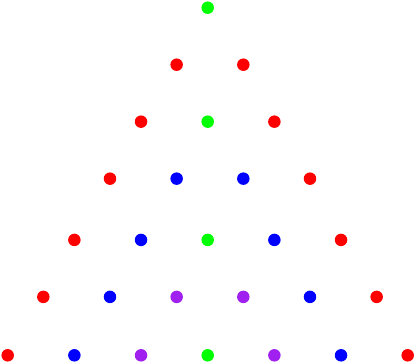}%
\caption{}
\label{7rowscolored}%
\end{center}
\end{figure}

We will extend this coloring scheme to ultimately show that $\lim_{n\to\infty} \frac{f(n)}{n} \leq \frac13$, and as a device, we will use an infinitely long stripe of the triangular lattice which has $n$ rows.  Denote this structure $S_n$, where for reference, $S_4$ is given in Figure \ref{s4}.\\

\begin{figure}[h]
\begin{center}
\includegraphics[scale=.6]{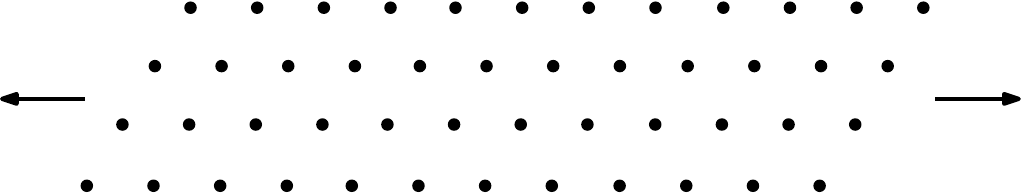}%
\caption{}
\label{s4}%
\end{center}
\end{figure}

Denote by $g(n)$ the minimum colors that can be used in a coloring of $S_n$ that forbids monochromatic equilateral triangles.  Clearly, $g(n) \leq 2f(n)$, as the points of $S_n$ can be partitioned into alternating copies of $T_n$ (which could then be properly colored with colors $c_1, \ldots, c_{f(n)}$) and inverted copies of $T_n$ (which could then be properly colored with colors $d_1, \ldots, d_{f(n)}$).  Just as the case with $f(n)$, exactly pinning down $g(n)$ seems difficult for all but the smallest values of $n$.  In some instances, however, we were able to give useful upper bounds.  For example, a proper 4-coloring of $S_6$ was obtained.  This was done by creating the base block given in Figure \ref{s6baseblock}, which is then repeatedly translated to the right and left.\\

\begin{figure}[hbt!]
\begin{center}
\includegraphics[scale=.8]{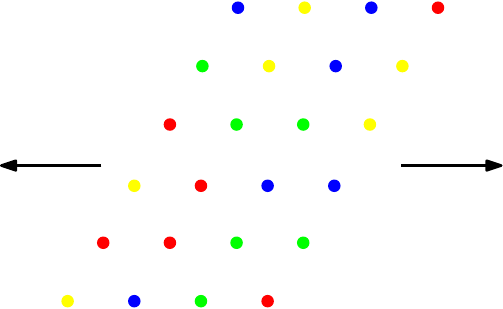}%
\caption{}
\label{s6baseblock}%
\end{center}
\end{figure}

To see that this coloring scheme does indeed yield a proper coloring of $S_6$, assume that a unit length is the minimum distance between pairs of points in $S_n$, and then note that the largest equilateral triangle that could possibly be formed with vertices in $S_6$ has edge length 5.  It follows that, if a monochromatic equilateral triangle were to exist in this coloring of $S_6$, it would have each of its three vertices lying in a window of three copies of the original base block (illustrated in Figure \ref{s6repeats}).  The authors used a computer search to discount the existence of a monochromatic equilateral triangle in this coloring, but this is not wholly unreasonable to verify by hand.\\   

\begin{figure}[h]
\begin{center}
\includegraphics[scale=.8]{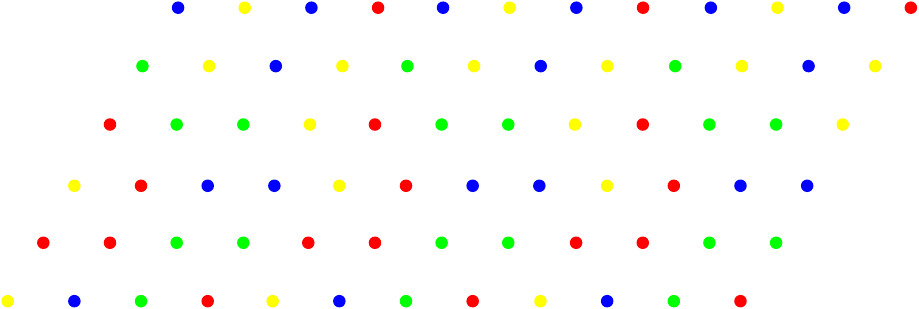}%
\caption{}
\label{s6repeats}%
\end{center}
\end{figure}

Our plan is to now create a new coloring scheme for $T_n$ by modifying the scheme referenced above in Figure \ref{7rowscolored}.  Instead of letting the middle column of points constitute one color class, we will use $d$ colors for $d$ middle columns, with the value of $d$ to be described below.  Other color classes will be essentially replaced by copies of $S_6$, as is visualized in Figure \ref{diagram}.\\

\begin{figure}[h]
\begin{center}
\includegraphics[scale=.8]{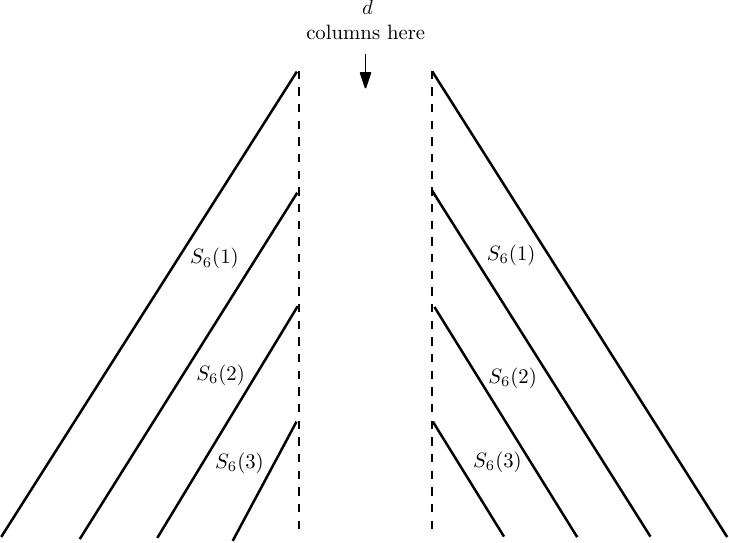}%
\caption{}
\label{diagram}%
\end{center}
\end{figure}

In Figure \ref{diagram}, $S_6(1)$ denotes a copy of $S_6$ properly colored with colors $\{c_{1,1}, c_{1,2}, c_{1,3}, c_{1,4}\}$, $S_6(2)$ denotes a copy of $S_6$ properly colored with colors  $\{c_{2,1}, c_{2,2}, c_{2,3}, c_{2,4}\}$, and so on.  To ensure that the coloring is proper, we need to select a large enough $d$ so that the two separate copies of $S_6(i)$ are far enough apart, guaranteeing that it is impossible to draw an equilateral triangle having two vertices in one copy and a third vertex in the other copy.  The existence of such $d$ is shown in Lemma \ref{existenceofd}, where it is important to note that, should we desire to color $T_n$ with some other $S_k$ in place of $S_6$, the corresponding $d$ could be chosen in terms of $k$ and independent of $n$.

\pagebreak

\begin{lemma} \label{existenceofd} Let $m$, $n$ be positive numbers with $m > \frac43n$. Let $R_1$ and $R_2$ be
infinite regions in the plane defined as follows. Region $R_1$ is bounded to
the right by the $y$-axis, above by the line $y = \sqrt{3}x$, and below by the line
$y = \sqrt{3}x - n$. Region $R_2$ is bounded to the left by $x = \frac{\sqrt{3}}{2}m$, above by the line $y = -\sqrt{3}(x-\frac{\sqrt{3}}2m)$, and below by the line $y = -\sqrt{3}(x-\frac{\sqrt{3}}2m)-n$. Then no equilateral triangle can be drawn with vertices in both $R_1$ and $R_2$.
\end{lemma}  

\begin{proof} Due to symmetry across the line $x=\frac{\sqrt3}{4}m$, it suffices to show that if an equilateral triangle has two vertices $v_1$ and $v_2$ in region $R_1$, its third vertex $v_3$ cannot be in region $R_2$. Let $v_1=(x_1,y_1)$ and $v_2=(x_2,y_2)$ be points in $R_1$. Without loss of generality we may assume that $y_1 \leq y_2$.

Given vertices $v_1$ and $v_2$, there are two possible locations for the third vertex $v_3$: we can rotate the point $(x_2,y_2)$  around $(x_1,y_1)$ by $60^\circ$ in a clockwise or a counterclockwise direction. If we rotate the region $R_1$ by $60^\circ$ in a counterclockwise direction, the resulting region will be bounded (in part) by two lines whose slope is $-\sqrt{3}$. These lines are parallel to the lines that bound $R_2$, and $v_3$ cannot in this case lie in $R_2$.

Now rotate the region $R_1$ by $60^\circ$ in a clockwise direction about $v_1$ and denote by $R'_1$ the image of $R_1$ under this transformation. Note that $v_3$ must lie in $R'_1$. We will show that $R'_1 \cap R_2$ is empty and, as a result, $v_3 \notin R_2$. The top rightmost point of $R'_1$, which we notate as $(x_T,y_T)$, can be found by rotating the point (0,0) $60^\circ$ clockwise around $(x_1,y_1)$:

$$
 \begin{bmatrix}
    x_T \\
    y_T
  \end{bmatrix}
=
  \begin{bmatrix}
    \frac12 & \frac{\sqrt{3}}{2} \\
     -\frac{\sqrt{3}}{2} & \frac12
  \end{bmatrix} 
 \begin{bmatrix}
    0 - x_1 \\
    0 - y_1
  \end{bmatrix}
    +
   \begin{bmatrix}
    x_1 \\
    y_1
  \end{bmatrix}
    =
    \begin{bmatrix}
    \frac12x_1 -\frac{\sqrt{3}}{2}y_1 \\
    \frac{\sqrt{3}}{2}x_1 + \frac12y_1
  \end{bmatrix} $$

\noindent Furthermore, the line $y=\sqrt{3}x$ gets rotated to the horizontal line $y=\frac{\sqrt{3}}{2}x_1+\frac12y_1$  and the line $y=\sqrt{3}x-n$ gets rotated to the line $y=\frac{\sqrt{3}}{2}x_1+\frac12y_1 - \frac12n$.

Consider the line $f(x)=-\sqrt{3}(x-\frac{\sqrt{3}}{2}m)-n$, which bounds $R_2$ from below. We will show that $(x_T,y_T) \notin R_2$ by showing that $f(x_T) > y_T$. 
\begin{eqnarray}
f(x_T) &=& -\sqrt{3}\left(x_T-\frac{\sqrt{3}}{2}m\right)-n \nonumber \\
 & =& -\sqrt{3}\left(\frac12x_1 -\frac{\sqrt{3}}{2}y_1-\frac{\sqrt{3}}{2}m\right)-n \nonumber \\
 & =& -\frac{\sqrt{3}}{2}x_1 +\frac{3}{2}y_1+\frac{3}{2}m - n \nonumber \\
 & =& \frac{\sqrt{3}}{2}x_1 + \frac12y_1-\sqrt{3}x_1 + y_1+\frac{3}{2}m - n \nonumber \\ 
 &=& y_T - \sqrt{3}x_1 + y_1+\frac{3}{2}m - n  \nonumber
\end{eqnarray}

\noindent Since the line $y=\sqrt{3}x - n$ bounds the region $R_1$ from below, it follows that $-\sqrt{3}x_1 \geq -y_1 - n$. This gives us the following inequality:

$$ f(x_T)\geq y_T -y_1 - n + y_1+\frac{3}{2}m - n = y_T + \frac{3}{2}m - 2n $$

\noindent By using the fact that $m > \frac43n$, we obtain the desired strict inequality $f(x_T) > y_T$.\qed
\end{proof}

A few additional comments should be made regarding the coloring scheme illustrated in Figure \ref{diagram}.  This construction shows that $\lim_{n\to\infty} \frac{f(n)}{n} \leq \frac{g(k)}{2k}$ for any $k \in \mathbb{Z}^+$.  We used $S_6$ (as opposed to some other $S_k$) because $\frac{g(6)}{6} \leq \frac23$, and that upper bound is the best we have been able to produce.  We have found no $k$ where we are able to prove that $\frac{g(k)}{k} < \frac23$, either by explicit coloring or some other argument.  Note also that, since $g(k) \leq 2f(k)$, any $k$ that is found where $\frac{f(k)}{k} < \frac13$ will also give us a better upper bound than the one we have at present.  Returning to the chart, given earlier in this section, listing upper bounds we have found for $f(n)$ for small $n$, one may notice $f(15) \leq 5$, which also implies $\lim_{n\to\infty} \frac{f(n)}{n} \leq \frac13$.  Again, we remark that this is the best we have been able to do, and we leave further improvements as an open question.

\begin{question} \label{limitproblem} Does there exist $k$ such that $\frac{f(k)}{k} < \frac13$?  In general, what is the value of the limit $\lim_{n\to\infty} \frac{f(n)}{n}$?
\end{question}

\vspace{50pt}

\section{A Counting Problem} 

All work done in this section will hinge on a straightforward observation, which is visualized in Figure \ref{pairsofpoints} for the case of $T_4$.  Given points $p_1$ and $p_2$ of $T_n$, there may be two possible options for a point $p_3$ in $T_n$ which allow $p_1, p_2, p_3$ to form the vertices of an equilateral triangle (see $p_1, p_2$ highlighted in red in Figure \ref{pairsofpoints}), or there may be one possible option for such a point $p_3$ (see $p_1, p_2$ highlighted in blue), or there may be no such point $p_3$ at all (see $p_1, p_2$ highlighted in green).\\

\begin{figure}[h]%
\begin{center}
\includegraphics[scale=.8]{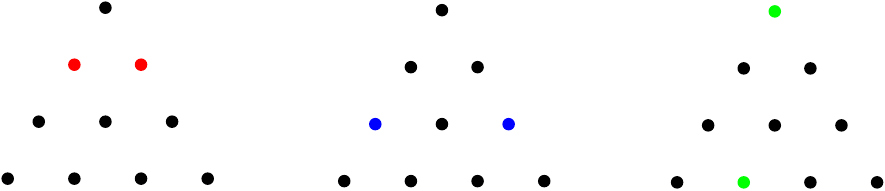}%
\caption{}
\label{pairsofpoints}%
\end{center}
\end{figure}

Let the respective functions $a_0(n)$, $a_1(n)$, and $a_2(n)$ count the number of pairs of points in $T_n$ that are vertices of zero, one, or two equilateral triangles in $T_n$.  As a quick warmup exercise, a reader may want to calculate for themselves that the ${10 \choose 2} = 45$ pairs of points of $T_4$ are sorted as $a_0(4) = 9$, $a_1(4)$ = 27, and $a_2(4) = 9$.  In what follows, we derive explicit formulas for $a_0(n)$, $a_1(n)$, and $a_2(n)$, with an interesting result being that $a_0(n) = a_2(n)$ for all $n$.

We will first determine $a_2(n)$.  Let $R_0$ be the rhombus drawn in Figure \ref{rhombusfigure}.\\  

\begin{figure}[h]%
\begin{center}
\includegraphics[scale=.8]{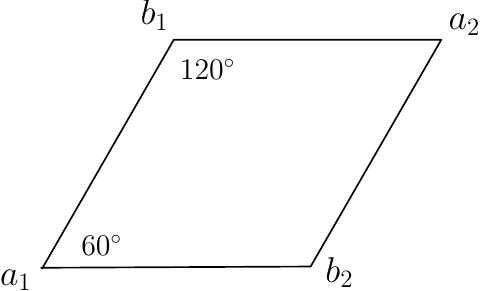}%
\caption{}
\label{rhombusfigure}%
\end{center}
\end{figure}

Let $\mathcal{R}$ consist of all rhombi similar to $R_0$.  Note that $a_2(n)$ is equivalent to the number of 4-element subsets of $T_n$ which constitute the vertices of some $R \in \mathcal{R}$, for the simple reason that pairs of points like those labeled $b_1, b_2$ in Figure \ref{rhombusfigure} exactly correspond to pairs counted by $a_2(n)$.  Define $R \in \mathcal{R}$ to be \textit{minimally contained} in some $T_k$ if $R$ can be drawn with its vertices being points of $T_k$, but $R$ cannot be so drawn in $T_{k-1}$.  Let the function $m(k)$ count the number of distinct $R \in \mathcal{R}$ that are minimally contained by $T_k$.  Our plan will now be to construct a specific formula for $m(k)$, and then for each $k \in \{3, \ldots, n\}$, multiply $m(k)$ by the number of times $T_k$ appears in $T_n$.  Taking the sum of all these individual products will give us the desired value $a_2(n)$.  Note here that we only need to do this for $k \in \{3, \ldots, n\}$ as $a_2(1)$ and $a_2(2)$ are both equal to 0.

We begin with a lemma describing how a minimally contained copy of $R \in \mathcal{R}$ must appear in some $T_k$.

\begin{lemma} \label{midpointtheorem} Let $R \in \mathcal{R}$ be minimally contained in $T_k$, and denote by $b_1, b_2$ the two vertices of the $120^{\circ}$ angles of $R$.  Then at least one of $b_1, b_2$ must be the midpoint of one of the sides of $T_k$.
\end{lemma}

\begin{proof} Consider Figure \ref{gastondiagram} where an arbitrary equilateral triangle has been drawn with midpoint $M$ on its bottom side and points $x,y$ placed so that the vectors $\overrightarrow{Mx}$ and $\overrightarrow{My}$ are perpendicular to the other two sides of the triangle.  These vectors are equal in length and the angle between them is $120^{\circ}$.  If we were to then fix point $M$, then rotate both $\overrightarrow{Mx}$ and $\overrightarrow{My}$ through some angle $\theta$ with $-30^{\circ} \leq \theta \leq 30^{\circ}$, and extend them so that they terminate at points $x'$ and $y'$ on the sides of the triangle, vectors $\overrightarrow{Mx'}$ and $\overrightarrow{My'}$ would also have the same length.

\begin{figure}[h]%
\begin{center}
\includegraphics[scale=.8]{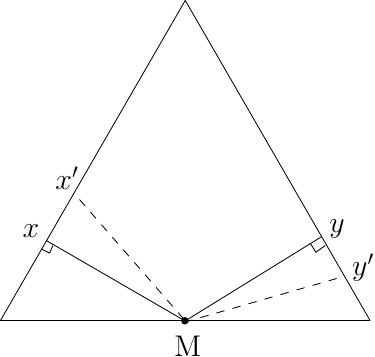}%
\caption{}
\label{gastondiagram}%
\end{center}
\end{figure}     

With the stipulation that $R \in \mathcal{R}$ is minimally contained in $T_k$, at least three of the four vertices of $R$ must be at points on the edges of $T_k$.  Without loss of generality, we may then assume that $b_1$ is placed on the bottom edge of $T_k$.  Note that, if $b_1$ were not placed at the midpoint of this edge, we could translate $a_1$, $b_1$, and $a_2$ to either the left or right to attain the configuration of points in Figure \ref{gastondiagram} with $a_1$, $b_1$, and $a_2$ taking place of $x'$, $M$, $y'$, respectively.  It then follows that, if $b_1$ was not originally placed at the midpoint of the bottom edge, we have the contradiction of either $a_1$ or $a_2$ lying outside of $T_k$.\qed
\end{proof}

\begin{lemma} \label{mofncount} If $k$ is odd, $m(k) = \frac32(k-1)$.  If $k$ is even, $m(k) = 0$. 
\end{lemma}

\begin{proof} Let $R \in \mathcal{R}$.  Lemma \ref{midpointtheorem} guarantees that if $R$ is minimally contained in $T_k$, at least one of the vertices $b_1$ and $b_2$ must be placed at the midpoint of some side of $T_k$.  If $k$ is even, the midpoint of each side of $T_k$ is not itself a point of $T_k$, so in this case, we have $m(k) = 0$.  Now assume $k$ is odd.  Note that there are three such $R$ where both of $b_1, b_2$ are placed at midpoints of sides.  We will now consider the set $S$ consisting of all $R$ having $b_1$ placed at the midpoint $M$ of the bottom side of $T_k$, vertices $a_1$ and $a_2$ being placed somewhere on the respective left and right sides of $T_k$, and $b_2$ placed somewhere in the interior of $T_k$, and once $|S|$ has been determined, multiply the result by three.

Starting from the bottom left corner of $T_k$ and continuing to the top corner, denote in order the points of the left side of $T_k$ as $p_1, \ldots, p_k$.  Note that the points $p_i \in \{p_1, \ldots, p_k\}$ where $p_i$ is a vertex of some $R \in S$ are exactly those with $i \in \{2, \ldots, \frac{k-1}{2}\}$.  It then follows that $|S| = \frac{k-3}{2}$, and in total, $m(k) = \frac{3(k-3)}{2} + 3 = \frac32(k-1)$.\qed
\end{proof}

We now develop an explicit formula for $a_2(n)$.  To do so, we will combine the previous two lemmas with well-known sums of the first $k$ positive integers, squares, and cubes, which for easy reference are given below.

\begin{center}
$\ds \sum_{j=1}^{k} j = \frac{k(k+1)}{2}$ \,\,\,\,\,\,\,\,\,\, $\ds \sum_{j=1}^{k} j^2 = \frac{k(k+1)(2k+1)}{6}$  \,\,\,\,\,\,\,\,\,\, $\ds \sum_{j=1}^{k} j^3 = \frac{k^2(k+1)^2}{4}$
\end{center}

\begin{theorem} \label{paircounttheorem} If $n$ is odd, $a_2(n) = \frac{(n-1)^2(n+1)(n+3)}{32}$.  If $n$ is even, $a_2(n) = \frac{n(n-2)(n+2)^2}{32}$. 
\end{theorem}

\begin{proof}  Let $k \in \{3, \ldots, n\}$, and denote by $h(k)$ the number of times $T_k$ appears in $T_n$.  Determination of $h(k)$ is a relatively straightforward matter as $T_k$ appears once in $T_n$ with its bottom side on the $k^{th}$ row of $T_n$, twice with its bottom side on the $(k+1)^{th}$ row of $T_n$, and so on.  This gives us $h(k) = 1 + 2 + \ldots + n - k + 1 = \frac{(n-k+1)(n-k+2)}{2}$.

First, we consider the case when $n$ is odd.  From our previous discussion, we have $\ds a_2(n) = \sum_{k=3}^{n} h(k)m(k)$.  Since $m(k) = 0$ for all even $k$, we may set $k = 2j+1$, and then rewrite this sum as  $\ds a_2(n) = \sum_{j=1}^{\frac{n-1}{2}} h(2j+1)m(2j+1)$, which in turn can be rewritten as $\ds  a_2(n) = \frac32 \sum_{j=1}^{\frac{n-1}{2}} (n-2j)(n-2j+1)j$.  It is somewhat tedious to verify, but expanding $ (n-2j)(n-2j+1)j$ and then applying the above formulas for $\ds \sum_{j=1}^{\frac{n-1}{2}} j^i$ with $i \in \{1,2,3\}$ yields $a_2(n) = \frac{(n+3)(n+1)(n-1)^2}{32}$.  

Now let $n$ be even.  We have $\ds a_2(n) = \sum_{k=3}^{n} h(k)m(k)$, and just as before, $m(k) = 0$ for all even $k$, so $a_2(n) = h(3)m(3) + h(5)m(5) + \cdots + h(n-1)m(n-1)$.  Write $k = 2j+1$ and reindex the previous sum to obtain $\ds a_2(n) = \sum_{j=1}^{\frac{n-2}{2}} h(2j+1)m(2j+1)$, where the only difference from the odd $n$ case is the sum's upper bound for $j$.  Evaluating functions $h$ and $m$ along with the previously mentioned $\ds \sum_{j=1}^{\frac{n-2}{2}} j^i$ with $i \in \{1,2,3\}$ gives us $a_2(n) = \frac{n(n-2)(n+2)^2}{32}$.\qed   
\end{proof}

\begin{theorem} \label{paircounttheorempart2} If $n$ is odd, $a_1(n) = \frac{(n^2-1)(n^2 + 2n + 3)}{16}$.  If $n$ is even, $a_1(n) = \frac{n(n+2)(n^2 + 2)}{16}$. 
\end{theorem}

\begin{proof} As mentioned in Section 1, the total number of 3-element subsets of $T_n$ that constitute the vertices of an equilateral triangle is given by $\alpha(n) = \frac{n^4 + 2n^3 - n^2 - 2n}{24}$.  It then follows that $\frac{n^4 + 2n^3 - n^2 - 2n}{8} = a_1(n) + 2a_2(n)$.  An application of Theorem \ref{paircounttheorem} and a small amount of algebra yields the desired result.\qed
\end{proof}

\begin{theorem} \label{paircounttheorempart3} For all $n$, $a_0(n) = a_2(n)$. 
\end{theorem}

\begin{proof} In Section 1 (and in the proof of the previous theorem), the total number of 3-element subsets of $T_n$ that constitute the vertices of an equilateral triangle is given by $\alpha(n) = \frac{n^4 + 2n^3 - n^2 - 2n}{24}$.  Let $\gamma(n)$ denote the total number of pairs of points in $T_n$, and note that $\gamma(n) = {{|T_n|}\choose{2}} = \frac{n^4 + 2n^3 - n^2 - 2n}{8}$.  Since $\gamma(n) = 3\alpha(n)$, each pair of points of $T_n$ appears together, on average, as the vertices of one equilateral triangle.  It follows that $a_0(n) = a_2(n)$.\qed 
\end{proof}

As remarked at the beginning of this section, it was quite surprising to the authors that $a_0(n) = a_2(n)$ for all $n$.  It strikes us that there should be some bijective argument demonstrating that $a_0(n) = a_2(n)$, one that does not require having an explicit count of all equilateral triangles in $T_n$.  We end this section with that question.

\begin{question} \label{equalityproblem} How does one construct a bijection between the set of all pairs of points of $T_n$ counted by $a_0(n)$ and those counted by $a_2(n)$? 
\end{question}

\section{Further Work}   

One of the oldest constructions in combinatorics is the \textit{Steiner triple system}.  Given a set $\mathcal{X}$ of $v$ points, it consists of a collection of 3-element subsets of $\mathcal{X}$ with the property that each pair of points of $\mathcal{X}$ appears in exactly one of these subsets.  Although triple systems were named after Jakob Steiner, who investigated them in the 1850s, it was shown by Kirkman in 1847 that such a system exists if and only if $v \equiv 1 \text{ or } 3 \pmod 6$.  For a historical perspective, see \cite{lindner}.  

Widely studied in the literature is the extension of a Steiner triple system to the concept of a \textit{pairwise balanced design}, often denoted as $(v, k, \lambda)$ and consisting of $v$ points and a collection of $k$-element subsets, here with the stipulation that each pair of points appears in exactly $\lambda$ of those subsets.  Note that using this notation, the original formulation of a Steiner triple system would be expressed as $(v, 3, 1)$.  Pairwise balanced designs are extended further to the notion of \textit{generalized block designs}, but the fundamental property of this family of structures is that once a collection of subsets has been determined, it is required that all pairs of points appear in the same number of those subsets.

While investigating topics in the previous two sections, it occurred to the authors that the 3-element subsets of $T_n$corresponding to the vertices of equilateral triangles constitute a modified version of a triple system, one where the parameter $\lambda$ has a different interpretation.  We propose this as an avenue for future study.

\begin{question} \label{modifiedstsproblem} For which $n, r$ is it possible to construct a collection of 3-element subsets of $\mathcal{X} = \{1, \ldots, n\}$ such that among all pairs of points of $\mathcal{X}$, exactly $r$ of those pairs appear in none of the 3-element subsets, exactly $r$ pairs appear in two of those subsets, and all other pairs appear in exactly one subset?    
\end{question}  

Note that if $r = 0$, Question \ref{modifiedstsproblem} is just asking for a Steiner triple system.  If $n$ is a triangular number with $r = a_2(n)$ as defined in Section 3, our observations on equilateral triangles in $T_n$ guarantee that such a structure exists.  In general, though, we cannot find mention of this question in the literature, and it seems like a nice topic for future investigation.

%%%%%%%%%%%% References %%%%%%%%%%%%%%%%
%\pagebreak

\pagebreak

\end{document}